\documentclass[12pt]{amsart}
\usepackage{amssymb, amsmath, amsthm}

\newtheorem{theorem}{Theorem}[section]

\newtheorem{proposition}[theorem]{Proposition}

\theoremstyle{definition}

\newtheorem{remark}[theorem]{Remark}
\newtheorem{example}[theorem]{Example}

\newtheorem{hypothesis}{Hypothesis}

\author[S. Peszat]{Szymon Peszat}
\date{December 07, 2021}
\address{\noindent Szymon Peszat, 
\newline \indent Faculty of Mathematics and Computer Science \newline \indent 
Jagiellonian University \newline \indent{\L}ojasiewicza  6 \newline
\indent 30-348 Krak{\'o}w, Poland } 

\email{szymon.peszat@im.uj.edu.pl}

\author[D. Zawisza]{Dariusz Zawisza}
\address{\noindent Dariusz Zawisza, 
\newline \indent Faculty of Mathematics and Computer Science \newline \indent 
Jagiellonian University \newline \indent{\L}ojasiewicza  6 \newline
\indent 30-348 Krak{\'o}w, Poland }

\email{dariusz.zawisza@im.uj.edu.pl}

\def\d{{\rm d}}
\def\e{{\rm e}}

\numberwithin{equation}{section}

\begin{document}

\baselineskip=17pt
\title{The investor problem based on the HJM model}

\thanks{The work  was supported by Polish National Science Center grant  2017/25/B/ST1/02584}


\maketitle

\begin{abstract}
We consider a   consumption-investment problem (both on finite and infinite time horizon) in which  the investor has an access to the bond market. In our approach prices of bonds with different maturities are described by the general HJM factor model. We assume that the bond market consists of the whole family of rolling bonds and the investment strategy is a general signed measure distributed on all real numbers representing time to maturity specifications for different rolling bonds. In particular, we can consider a portfolio of coupon bonds. The investor's objective is to maximize time-additive utility of the consumption process. We solve the problem by means of the HJB equation for which we prove required regularity of its solution and all required estimates to ensure applicability of the verification theorem. Explicit calculations for affine models are presented.
\end{abstract}

\section{Introduction}

The famous Merton problem of maximizing expected utility of a consumption stream has a long tradition and has been solved under many different settings. In the original formulation, see e.g. \cite{Oksendal-Sulem}, we have a market with  two possible investments: a safe investment (bank account) with the price dynamics 
$$
\frac{\d B(t)}{B(t)}=r\d t, \quad B(s)=b
$$
and a risky investment (stocks) with the price dynamics
$$
\frac{\d P(t)}{P(t)}= \mu \d t+ \sigma \d W(t), \quad P(s)= p, 
$$
where $r$,  $\mu$ and $\sigma$ are constants, and $W$ is a real valued  Wiener process.  The investor at any time $t$ chooses a  consumption rate $C(t)$ and can transfer money from one investment to the other without transaction costs. Then the dynamic of the corresponding capital  $z^{\psi,C}$ is given by 
$$
\frac{\d z^{\psi,C}(t)}{z^{\psi,C}(t)}= \left[r(1-\psi_t) + \mu \psi_t-C(t)\right]\d t+ \psi_t\sigma\d W(t), \quad z^{\psi,C}(s)= z,
$$
where $\psi_t$ is the fraction of the capital invested in stocks at time $t$. Given a horizon of the investments $T$,   a discount factor $\gamma\ge 0$, $\alpha \in (-\infty,0)\cup (0,1)$ 
and $a,b \geq 0$, the goal is to find a strategy $(\widehat \psi,\widehat C)$ which maximizes the satisfaction (or reward) functional 
\begin{equation}\label{E11}
\begin{aligned}
&J_T(z,s,\psi,C)\\
& = \frac{1}{\alpha}\mathbb{E}\left[ a\int_s^T\e ^{-\gamma(t-s)} \left(C(t)z^{\psi,C}(t)\right)^\alpha \d t + b\e^{-\gamma (T-t)}\left(z^{\psi,C}(T)\right)^{\alpha}\right].
\end{aligned}
\end{equation}

In the present paper an investor has an access to the market of the so-called  rolling  bonds $U(t)(x)$, $t\in [s,T]$ and $x\in [0,T^*]$. Here  $T^*$ is the maximal time to maturity. Given $x$, the rolling bond $U(\cdot)(x)$ is a   self-financing investment in the bond with fixed time to maturity $x$ and therefore is a tradable instrument (for the precise definition see Rutkowski \cite{Rutkowski} and also Appendix \ref{Appendix}). The dynamics of the rolling bonds is given by the following SDE
\begin{equation}\label{E12}
\frac{\d U(t)(x)}{U(t)(x)} =   \overline r(t) \d t   - \left\langle \tilde \sigma(t)(x), \lambda(t) \d t + \d W(t)\right\rangle,
\end{equation}
where $\overline r$ is a (random)  short rate, $\tilde \sigma$ is a random field,  $\lambda$ is a  random process taking values in $\mathbb{R}^m$ and $W$ is an  $m$-dimensional Wiener process, see \eqref{E31} and \eqref{E34} from Section \ref{S3}. 

The investment strategy $\psi_t$ is a normalized signed measure on $[0,T^*]$.   The dynamics of the corresponding capital  $z^{\psi, C}$ has the form
\begin{equation}\label{E13}
\frac{\d z^{\psi,C}(t)}{z^{\psi,C}(t)}=  \int_0^{T^*} \frac{\d U(t)(x)}{U(t)(x)} \psi_{t}(\d x) - C(t)\d t.
\end{equation}
The goal is to maximize the satisfaction (or reward) functional  
$$
J_T(u(s)(\cdot), z^{\psi,C}(s),s) 
$$
given by the right hand side of \eqref{E11}. We will restrict our attention to the  Heath--Jarrow--Merton  factor model with rolling bonds as the basic instrument. Recall that, in the HJM factor model, the short rate  can be factorized in the form $\widehat r(t)=\varphi(t,Y(t))$, where $Y$ is a diffusion on $\mathbb{R}^n$. This enables us to solve the problem by  means of the general HJB theory. In addition to the general theory we provide calculations for the Vasicek, Cox--Ingersoll--Ross and G2++ models. The last mentioned model has never been considered in the context of the optimal portfolio selection problem.  In our main result we show that the corresponding HJB equation has a solution, this solution admits a Feynman--Kac representation, and  that we can check  the hypotheses of  the general verification theorem.  Finally we provide formulae for the optimal consumption rate $\widehat C$ and investment strategy $\widehat \psi$.
 
  \bigskip
Now let us briefly recall the state of the literature and mark a few differences between our approach and the literature findings. There are many papers which have solved  the consumption-investment problem under the framework of the short rate dynamics and without taking care of other HJM models (see Wachter \cite{Wachter}, Guasoni and Wang \cite{Guasoni}, \cite{Wang}, Korn and Kraft \cite{Korn}, Chang and Chang \cite{Chang}, Detemple and Rindisbacher \cite{Detemple},  Deelstra et al. \cite{Deelstra}, Hata et al. \cite{Hata}, Synowiec \cite{Synowiec}, Trybu{\l}a \cite{Trybula}). Moreover, many of them are  focused more on particular interest rate models such as the Vasicek model or the Cox--Ingersoll--Ross model. We consider the general factor model with rather weak assumptions on coefficients regularity that includes many other short term interest rate models. We consider  finite and infinite time horizon problems. It should be noticed that the work of Guasoni and Wang \cite{Wang} has provided the solution for the infinite horizon incomplete market model and on domains for the stochastic factor being a subset of $\mathbb{R}^{n}$. At the same time Hata et al. \cite{Hata} have considered the finite horizon problem with general factor model with factor dynamics operating on entire $\mathbb{R}^{n}$. Note that both papers use different arguments than we do and do not take into account delicate nature of the bond market. 

Ringer and Tehranchi \cite{Ringer} and Palczewski \cite{Palczewski} have considered the problem of maximizing the utility of the terminal wealth in a large generality but without assuming any consumption stream and without introducing the concept of the rolling bond.

Many authors have  considered the so called rolling bonds under various objectives, but the framework is limited only either to finite number of rolling bonds and very specific affine factor dynamics (see the risk sensitive criterion of Bielecki et al. \cite{Bielecki}, Bielecki and Pliska \cite{Bielecki2})  or one rolling bond with one or two specific factors. Let us mention only few items dedicated to the pension management problem: Battochio and Menoncin \cite{Battochio}, Guan and Liang \cite{Guan}, \cite{Guan2}, Zhang et al. \cite{Zhang}. Apart from the general factor model we consider a general normalized signed measure as the  investor's strategy, which can be useful to deal with a changing number of bonds available on the market.

Munk and Sorensen \cite{Munk} have solved the consumption investment problem under the interest rate risk by the duality methods (see Cox and Huang \cite{Cox},  Karatzas et al. \cite{Karatzas}). Their solution brings some knowledge about the investment position in general HJM models. However, they have not used rolling bonds, have not considered infinite horizon case and have not presented any detailed solution in the affine framework.

It is worth mentioning that our analysis is conducted in the  stochastic environment with unbounded coefficients which allows us  to develop many technical contributions. For example, we present a novel method to prove the verification theorem on the infinite horizon. This should be compared to the different methods proposed for example by Nagai \cite{Nagai} and Guasoni and Wang \cite{Wang}.
 
 \bigskip
The paper is organized  as follows. In the next two sections we introduce  the concept of the controlled investment process and the HJM factor model. The main result for the finite horizon case  (Theorem \ref{T41}) is formulated and proved in Sections \ref{S4} and \ref{S5}, respectively.  Sections \ref{S6} and \ref{S7} are  devoted to the solution for some particular models. In Section \ref{S8}  we study the infinite horizon case, see Theorem \ref{T84}.  In  Appendix \ref{Appendix}  we recall basic definition and properties of the bond market.

\section{Controlled process}
Let $(\Omega, \mathfrak{F}, \mathbb{P})$ be a probability space,  and let $T^{*}\in (0,+\infty)$ be the  maximal time to maturity available on the market. Given a finite time  investment horizon $T\in (0,+\infty)$, we denote by   $\mathcal{M}_{T}$  the space of all (weakly) predictable processes $\psi$ taking values in the space of all signed measures with the finite total variation norm satisfying
\begin{align*}
&\int_0^{T^*} \psi_{t}(\d x)=1, \quad \forall\, t\in [0, T], 
\\
&\int_0^{T}\int_0^{T^*} \left\vert \langle \tilde \sigma(t)(x), \lambda(t)\rangle \right\vert  \|\psi_{t}\|_{\textrm{Var}}(\d x)\d t<+\infty,\\
&\int_0^{T}\left( \int_0^{T^*}  \left\|  \tilde \sigma(t)(x)\right\| \|\psi_{t}(\d x) \|_{\textrm{Var}}\right) ^2\d t < +\infty,
\end{align*}
where $\|\cdot \|_{\textrm{Var}}$ stands for the total variation norm. 

Let $\mathcal{C}_T$ be the space of all non-negative predictable  processes. We call  $\mathcal{A}_T= \mathcal{M}_T\times \mathcal{C}_T$ the set of \emph{admissible strategies}. Let $(\psi,C)\in \mathcal{A}_T$.   Let $z^{\psi, C}(t)$ denote  the capital of an investor, whose consumption rate at time $t$ is $C(t)z^{\psi,C}(t)$ and who can invest in (rolling) bonds with an investment  strategy $\psi$. Then $z^{\psi,C}$ satisfies \eqref{E13}. Combining \eqref{E12} and \eqref{E13}  we obtain 
\begin{equation}\label{E22}
\frac{\d z^{\psi,C}(t)}{z^{\psi,C}  (t)} = \left[  \overline r(t)   -C(t)\right]  \d t -\int_0^{T^*}\left\langle  \tilde \sigma(t)(x)\psi_{t}(\d x) ,\lambda(t)\d t +  \d W(t)\right\rangle. 
\end{equation}

Note that the measure $\psi_{t}$ admits negative values. In the following particular case 
\begin{equation} \label{maturities}
\psi_{t}=\sum_{i=1}^{n} \left[\eta_{1,t_{i}}\delta_{x_{1}}+\eta_{2,t_{i}}\delta_{x_2}+\ldots+\eta_{l_{i},t_{i}}\delta_{x_{t_{i}}}\right] \chi_{(t_{i},t_{i+1}]}(t),
\end{equation}
where $(\eta_{1,t_{i}},\eta_{2,t_{i}},\ldots \eta_{l_{i},t_{i}})$  are random vectors  such that  $\sum_{k=1}^{l_i} \eta_{k, t_i}=1$.  Each point mass measure $\delta_{x_k}$ corresponds with the rolling bond with the fixed time to maturity $x_k$. Note that in this setting we reflect the fact that in practice on each time segment $(t_{i},t_{i+1}]$ we have different numbers  and  types of maturities (bonds) available on the market. On the other hand this is a convenient way to place different coupon bonds in the portfolio. Namely, we can take
\[
\psi_{t}=\sum_{i=1}^{n} \left[\eta_{1,t_{i}}\psi_{t,1}+\eta_{2,t_{i}}\psi_{t,2}+\ldots+\eta_{l_{i},t_{i}}\psi_{t,l_{i}}\right] \chi_{(t_{i},t_{i+1}]}(t),
\]
where $\psi_{t,1},\psi_{t,2},\ldots,\psi_{t,l_{i}}$,
represent different coupon bonds and are of the form \eqref{maturities}.

\section{HJM factor model}\label{S3}
 In this paper we restrict our attention to the so-called HJM factor model.  Namely, we assume that  the short rate has the form $\overline r(t)= \varphi(t, Y(t))$, where 
\begin{equation}\label{E31}
\d Y(t)=\left[B(t, Y(t)) +  \Sigma(t, Y(t)) \lambda(t, Y(t)) \right]\d t+ \Sigma(t, Y(t)) \d W(t),
\end{equation}
$B \colon [0,T] \times \mathbb{R}^{n} \to \mathbb{R}^{n} $,  
$\lambda \colon [0,T] \times \mathbb{R}^{n} \to \mathbb{R}^{m}$,  $\varphi \colon [0,T] \times \mathbb{R}^{n} \to \mathbb{R}$,  $\Sigma \colon  [0,T] \times \mathbb{R}^{n} \to \mathbb{R}^n\otimes \mathbb{R}^m$, and $W=(W_{1},W_{2},\ldots W_{m})$ is an $m$-dimensional  Wiener process  defined on a filtered probability space $(\Omega, \mathfrak{F},(\mathfrak{F}_t), \mathbb{P})$.  

\bigskip
We assume that the following hypotheses are satisfied. The first one is on the regularity of $B$, $\lambda$, and $\varphi$. For the existence and uniqueness  of a solution to SDE \eqref{E31} we need at least  a  local Lipschitz continuity in $y$, unfortunately, we need at least a H\"older continuity in $t$ and $y$ for the existence of a classical solution to certain PDEs with coefficients $B$, $\Sigma$, and $\varphi$, and their  Feynman--Kac representations, see e.g. Friedman \cite{Friedman}. 

\begin{hypothesis} \label{A1}\text{}

\noindent
{\it i)} The  mappings  $B$, $\lambda$ are bounded and locally Lipschitz continuous with respect to $t$ and   $y$, while  $\varphi$ is locally Lipschitz continuous  with respect to $t$ and $y$ and satisfies the linear growth condition in $y$ uniformly with respect to $t$. 

or

\noindent 
{\it ii)} The mappings  $B$, $\varphi$, $\lambda$   are   Lipschitz continuous  in $t$ and $y$. 
\end{hypothesis}

The second hypothesis  is on boundedness, Lipschitz continuity  and uniform ellipticity of $\Sigma$. This is our most restrictive assumption. In particular it does  not allow us to cover directly the important  Cox--Ingersoll--Ross model. In this case we treat the problem separately and present some explicit calculations together with other particular examples focused mainly on affine models.
\begin{hypothesis}\label{A2}
The mapping  $\Sigma$ is bounded, Lipschitz continuous  in $t$ and $y$,  and  there exists a constant   $c>0$ such that  
$$
\sum_{i,j=1}^n \left(  \Sigma(t,y)  \Sigma^{\top}(t,y)\right)_{i,j} \xi_i\xi_j \geq c \|\xi\|^{2}, \quad \xi,y \in \mathbb{R}^{n}, \;t \in [0,T].
$$
\end{hypothesis}
The last hypothesis is a no-arbitrage assumption. 
\begin{hypothesis}\label{A3}
$\lambda\colon [0,T]\times \mathbb{R}^{n} \to \mathbb{R}^{n}$ is a bounded,  locally Lipschitz continuous  mapping and 
$$
\d\mathbb{P}^{*} = \e^{ \int_{0}^T\langle  \lambda(t,Y(t)), \d W(t)\rangle  - \frac 12 \int_{0}^T \| \lambda(t,Y(t))\|^2 \d t}\d\mathbb{P} 
$$
is a martingale measure, see e.g. Appendix \ref{Appendix}. 
\end{hypothesis}
\begin{remark}
Hypothesis \ref{A3} gives the form of the volatility coefficient $\tilde \sigma$ appearing in \eqref{E12}. Indeed, the price $P(t,S)$ at time $t$ of the  zero coupon bond with maturity $S$ is  
$$
P(t,S)= \mathbb{E}^{*} \left\{ \e^{ -\int_{t}^{S}\varphi(u,Y(u)) \d u }\vert \mathfrak{F}_t\right\}
=
\mathbb{E}^{*} \left\{ \e^{-\int_{t}^{S}\varphi(u,Y(u)) \d u }\vert Y(t)\right\}. 
$$
Hence $P(t,S)=F(t,S,Y(t))$,  where $F$ is a continuous function of $t\in [0,T]$, $t\le S\le T^*$. Moreover, for fixed $S$,  $F(\cdot,S,\cdot)\in C^{1,2}([0,S]\times \mathbb{R}^n)$, is the unique  solution to the following equation
\begin{equation}\label{E32}
\frac{\partial } {\partial t} F(t,S,y)+ L_0F(t,S,y)- \varphi(t,y)F(t,S,y)=0, \qquad  F(S,S,y)=1,
\end{equation}
where 
\begin{equation}\label{E33}
\begin{aligned}
&L_0F(t,S,y)\\
&\quad := \frac{1}{2} \operatorname{Tr}\left(\Sigma(t,y) \Sigma^{\top}(t,y) D^{2}_{yy} F(t,S,y) \right) + \langle B(t,y), D_{y}F (t,S,y)\rangle. 
\end{aligned}
\end{equation}
Therefore, see Appendix A, the no-arbitrage condition yields
\begin{equation}\label{E34}
\tilde \sigma(t)(x)=  -\Sigma(t,Y(t))^{\top}  D_{y} \log F(t, t+x,Y(t)). 
\end{equation}
\end{remark}

Given a signed measure $\psi$ on $[0,T^*]$ set 
\begin{equation}\label{E35}
A(\psi)(t,y):= \Sigma(t,y)^{\top}  \int_0^{T^*}  D_{y} \log F(t,t+x,y)\psi(\d x). 
\end{equation}
 Let $\psi\in \mathcal{M}_T$. Taking into account  \eqref{E22},  \eqref{E34}, and \eqref{E35},  we infer  that  in the HJM factor model the wealth dynamics has     the form
\begin{equation}\label{E36}
\begin{aligned}
&\frac{\d z^{\psi,C}(t)}{z^{\psi,C}(t)}\\
&\quad = \left[\varphi(t,Y(t)) - C(t)\right] \d t + \left\langle A(\psi_t)(t,Y(t)), \lambda(t,Y(t)\d t + \d W(t)\right\rangle. 
\end{aligned}
\end{equation}

\bigskip
In the present paper we consider optimal finite and infinite horizon consumption problems.   In the finite horizon case the objective of the investor is to maximize
\begin{equation}\label{E37}
\begin{aligned}
&J_T(z,y,s,\psi,C) \\
&\quad = \frac{1}{\alpha}\mathbb{E}\left[ a\int_s^T\e ^{-\gamma(t-s)} \left(C(t)z^{\psi,C}(t)\right)^\alpha \d t + b\e^{-\gamma (T-t)}\left(z^{\psi,C}(T)\right)^{\alpha}\right], 
\end{aligned}
\end{equation}
where $z$ and $y$ are values of $z^{\psi, C}$ and  $Y$ at a given   initial time $s\in [0,T]$, $\gamma\ge 0$ is a discount factor, $\alpha \in  (0,1)$ and $a,b \geq 0$. Therefore, the investor can decide how to distribute his preference between the consumption stream and the terminal weight by controlling the parameters $a$ and $b$. In fact some of our results hold true also for $\alpha <0$, see Remark \ref{R}. Let 
$$
V_T(z,y,s):=\sup_{(\psi,C)\in \mathcal{A}_T} J_T(z,y,s,\psi,C)
$$
be the \emph{value function}.  Our main result concerning the finite horizon problem is Theorem \ref{T41} below. 

In the infinite horizon case we assume that the functions $B$, $\Sigma$, $\lambda$ and $\varphi$ do not depend on the time variable. The reword functional is of the form 
\begin{equation}\label{E38}
J(z,y,\psi,C)= \frac{1}{\alpha}\mathbb{E}\, \int_0^{+\infty}\e ^{-\gamma t} \left(C(t)z^{\psi,C}(t)\right)^\alpha \d t. 
\end{equation}
 Above $z$ and $y$ are values of $z^{\psi,C}$ and $Y$ at initial time $0$.  Our main results concerning the infinite horizon case are formulated and proved in Section \ref{S8}. 
\section{Main result concerning finite horizon problem}\label{S4}
Recall that $L_0$ is defined by \eqref{E33}. Let 
\begin{equation}\label{E41}
Lf (t,y)= L_0f(t,y)+ \left\langle \Sigma(t,y)\lambda(t,y) ,  D_{y}f(t,y)\right\rangle
\end{equation}
be the generator of the diffusion \eqref{E31}.  Let 
\begin{equation}\label{E42}
g(t,y):= \frac{1}{1-\alpha}\left[\alpha \varphi (t,y)  + \frac{\alpha}{2(1-\alpha)} \| \lambda(t,y) \|^2   -\gamma\right]. 
\end{equation}
 Consider the following linear PDE
\begin{equation} \label{E43}
\frac{\partial G}{\partial t} + LG+\frac{\alpha}{(1-\alpha)}   \left\langle \Sigma \lambda , D_{y} G\right\rangle+ g G  +a^{\frac{1}{1-\alpha }} =0,\quad 
G(T,y)= b^{\frac{1}{1-\alpha}}. 
\end{equation}

For the existence of an optimal strategy we will need the following hypothesis. 
\begin{hypothesis}\label{A4}
The fraction $\frac{D_{y} G(t,y)}{G(t,y)}$ is globally bounded and there is a weakly measurable mapping $\widehat \psi\colon [0,T]\times \mathbb{R}^n\to  \mathcal{M}_T$ such that for all $t\in [0,T]$ and $y\in \mathbb{R}^n$, 
 $$
 A(\widehat \psi_{t,y})(t,y)= \frac {\lambda (t,y)}{(1-\alpha)}  + \Sigma^{\top}(t,y) \frac{D_{y} G(t,y)}{G(t,y)}. 
$$
 \end{hypothesis}
\begin{theorem} \label{T41} $(i)$ Assume  Hypotheses \ref{A1} to \ref{A3}. Then there exists  a  unique classical solution $ G \in C^{1,2}[0,T)\times \mathbb{R}^{n}) \cap C([0,T]\times \mathbb{R}^n)$ of \eqref{E43} satisfying the exponential growth condition $|G(t,y)| \leq A \e^{B\|y\|}$. Moreover, $G$ admits the following   Feynman--Kac representation
\begin{equation}\label{E44}
G(t,y) =\mathbb{E} \, a^{\frac{1}{1-\alpha}}\int_{t}^{T}\e^{  \int_{t}^{u}  g(k,\tilde Y(k)\d k}\d u  + \mathbb{E} \,   b^{\frac{1}{1-\alpha}}\e^{ \int_{t}^{T} g(k,\tilde Y(k)\d k} ,
\end{equation}
 where $g$ is given by \eqref{E42}, and $\widetilde{Y}$ solves 
\begin{equation}\label{E45}
\begin{aligned}
\d  \widetilde{Y}(k)&= \left[B(k, \widetilde{Y}(k)) + \frac{1}{(1-\alpha)}  \Sigma(k, \widetilde{Y}(k))    \lambda( k, \widetilde{Y}(k)) \right]\d k\\
&\qquad + \Sigma(k, \widetilde{Y}(k)) \d W(k), \\ 
\tilde Y(t)&=y. 
\end{aligned}
\end{equation}
$(ii)$ Assume that additionally Hypothesis \ref{A4} holds. Then 
$$
V_T(z,y,s)= G(s,y)^{1-\alpha}\e^{-\gamma s}z^\alpha= J_T(z,y,s, \widehat{\psi}, \widehat{C}), 
$$
 where the optimal  investment policy $\widehat \psi$ and the optimal consumption $\widehat C$ are given by  $\widehat \psi_t= \widehat{\psi}_{t,Y(t)}$ and $\widehat C(t)= G(t,Y(t))^{-1}$ for $t\in [s,T]$. 
\end{theorem}

Theorem \ref{T41} is proved in Section \ref{S5}. 
\subsection{Some auxiliary results}

\begin{remark} \label{R43}Given $t$ and $y$, consider the following equation for a signed measure $\psi$; 
\begin{equation}\label{E47}
A(\psi)(t,y)=\frac {\lambda (t,y)}{(1-\alpha)}  + \Sigma^{\top}(t,y) \frac{D_{y} G(t,y)}{G(t,y)}, \quad  \int_0^{T^*} \psi(\d x)=1. 
\end{equation}
This equation appears in Hypothesis \ref{A4}. The existence of its solution is crucial for the existence of an optimal investment strategy. Note that the right hand side of the first equation above is a vector in $\mathbb{R}^m$. Let  $\mathbf{x}= (x_1,\ldots, x_{l})\in [0,T^*]^{l}$, where $l$ is large enough.   We are looking for $\psi$ of the form 
$$
\psi=\sum_{k=1}^{l} \eta_k(t,y)\delta_{x_k}. 
$$
where $\eta(t,y)=(\eta _1(t,y), \ldots,\eta_{l}(t,y))^{\top} \in \mathbb{R}^{l}$.  Since  $A(\psi)(t,y)$ is given by \eqref{E35} we have the following system of linear equations for the column vector $\eta(t,y)$; 
$$
\left\{
\begin{aligned}
\mathcal{F}(t,y, \mathbf{x}) \eta(t,y)&=  \frac {\lambda (t,y)}{(1-\alpha)} \ + \Sigma^{\top}(t,y) \frac{D_{y} G(t,y)}{G(t,y)}, \\
\sum_{k=1}^{l}\eta_k(t,y) &=1,
\end{aligned}
\right.
$$
where $\mathcal{F}(t,y, \mathbf{x})$ is the  $m\times l$ matrix with columns
\[
\Sigma(t,y)^\top D_{y} \log F(t,t+x_{1},y), \ldots,  \Sigma(t,y)^\top D_{y} \log F(t,t+x_{m+1},y).
\]
Here the derivative $D_{y} \log F(t, t+x_{k},y)$ is understood as a column vector.  Let $e=(1,1,\ldots,1)$. The solution exists provided that for the sequence $\mathbf{x}$ the $(m+1)\times l$-matrix
$$
\overline{\mathcal{F}(t,y, \mathbf{x})}= 
 \left[ {\begin{array}{c}
  \mathcal{F}(t,y, \mathbf{x}) \\
   e \\
  \end{array} } \right] 
$$
has rank $(m+1)$.  Summing up, if there are $l$ and a vector $\mathbf{x}\in \mathbb{R}^l$ such that for all $t$ and $y$, $\overline{\mathcal{F}(t,y, \mathbf{x})}$ has rank $m+1$, then condition \eqref{E47} is fulfilled. Moreover, one can chose the optimal investment strategy in the form 
$$
\widehat\psi_t= \sum_{k=1}^{l} \eta_k(t,Y(t))\delta_{x_k}. 
$$
\end{remark}

\begin{remark}
Note that \eqref{E35}  has much in common with the standard \emph{duration} of the portfolio of bonds used frequently in the static bond portfolio immunization.  So we might say that Theorem \ref{T41} gives a recipe for the dynamic portfolio immunization.
\end{remark}

\section{Proof of  Theorem \ref{T41}}\label{S5}
 First of all note that \eqref{E43} is a linear equation and under Hypotheses  \ref{A1} and \ref{A3}  it has  a unique smooth classical solution ${C}^{1,2}([0,T)\times \mathbb{R}^{n} \times [0,T)) \cap {C}([0,T]\times \mathbb{R}^n)$, see Zawisza \cite[Theorem 3.3]{Zawisza},  such that $|G(y,t)| \leq A e^{B\|y\|}$ (it follows easily from Zawisza \cite[Lemma 3.2]{Zawisza} and the proof of Zawisza \cite[Theorem 3.3]{Zawisza}). Moreover, $G$ admits the  Feynman--Kac representation \eqref{E44} and \eqref{E45}. 

Assume that Hypothesis \ref{A4} is fulfilled. To solve the optimization problem we will use the HJB approach. As usual we will try to find the function $V$ in the form
\begin{equation*} 
V(z,y,s)=\frac{1}{\alpha} K(s,y)\e^{-\gamma s}z^\alpha.
\end{equation*}
Recall that $\alpha \in (0,1)$. For $\alpha<0$ one  would need to exchange $\sup$ with $\inf$ in the HJB equation. Let us write  the aforementioned  HJB equations for the function $K$, 
\begin{align*}
\frac{\partial K}{\partial t}&+ LK+(\alpha \varphi -\gamma) K + \sup_{c\ge 0}\left[ -\alpha K c + ac^\alpha\right]\\
&+\alpha \sup_{\psi} 
\left\{ \frac{ \|A(\psi)\|^2 (\alpha -1)}{2} K+  \langle A(\psi),  \lambda   K+ \Sigma^{\top} D_{y} K \rangle \right\}=0,
\end{align*}
with the terminal condition  $K(T,y)=b$. 

Note that we have
$$
 \sup_{c\ge 0}\left[ - \alpha K c +a c^\alpha\right]= a(1-\alpha)\left(\frac{K}{a}\right)^{\frac{\alpha}{\alpha -1}}
$$
and the supremum is attained at  $\widehat c= (K/a)^{\frac 1{\alpha -1}}$. 

Next note that 
\begin{multline*}
\alpha \sup_{A\in \mathbb{R}^m}\left\{ \frac{ \|A\|^2 (\alpha -1)}{2} K+  \langle A,  \lambda   K+ \Sigma^{\top} D_{y} K \rangle\right\}\\
= 
\frac{\alpha}{2(1-\alpha)}\frac{1}{K} \|\lambda K + \Sigma^{\top} D_{y} K\|^2
\end{multline*}
and the supremum equals 
$$
\overline A:= \frac 1{(1-\alpha )} \left( \lambda  + \Sigma^{\top} \frac{D_{y} K}{K}\right).  
$$
We will show that $K^{1-\alpha}(t,y)= G(t,y)$, therefore  Hypothesis \ref{A4} ensures that  given $t$ and $y$ there is a signed  measure $\widehat \psi_t(\d x)(y)$ such that 
\[
A(\widehat\psi_t(\cdot)(y))= \frac 1{(1-\alpha)} \left( \lambda(t,y)  + \Sigma^{\top}(t,y) \frac{D_{y} K(t,y)}{K(t,y)}\right), \;  \int_0^{T^*} \widehat \psi_t(\d x)(y)=1. 
\]

Hence,  we eventually arrive at  the  following HJB equation

$$
\begin{aligned}
0&= \frac{\partial K}{\partial t}+ LK+(\alpha \varphi -\gamma) K  + a(1-\alpha) \left( \frac{K}{a}\right)^{\frac {\alpha}{\alpha -1}} \\
&\qquad +\frac{\alpha}{2(1-\alpha)} \frac{1}{K} \left\| \lambda K + \Sigma^{\top} D_{y}K \right\|^{2},\\
b&= K(T,y). 
\end{aligned}
$$
The proof of the theorem will be completed as soon as we can show that: 
\begin{itemize}
\item $G(t,y):= K(t,y)^{1-\alpha}$ satisfies \eqref{E44}. 
\item We have 
$$
\frac {D_{y} K}{K}= (1-\alpha)\frac{{D_{y} G}}{G}. 
$$
\item We present the verification reasoning for the function $\e^{-\gamma t}K(t,y)z^\alpha$. 
\end{itemize}

An elementary verification of the first two items is left to the reader. Hypothesis \ref{A4} guarantees boundedness of $D_{y}G/G$.  Therefore, under Hypotheses \ref{A1}--\ref{A3},  the optimal state process can be rewritten as
$$
 z^{\hat{\psi},\hat{C}}(t)=z\e^{\int_{s}^{t} h(\widetilde{Y}(u))\d u}Z(t),
$$
where $h$ satisfies the linear growth condition, while $Z$ is a square integrable martingale. Next, we can use the fact that
\begin{equation} \label{uniform}
\mathbb{E}\sup_{s \leq t \leq T} \e^{A \|\widetilde{Y}(t)\|} < +\infty, 
\end{equation} 
see Zawisza \cite[Lemma 3.2]{Zawisza}, which guarantees the uniform integrability condition for certain family of random variables and ensures that we can use the  verification theorem to prove that $(\widehat{\psi}, \widehat{C})$ is an optimal control. 
More precisely,  applying the It\^{o} formula  and taking the expectation of both sides, we get
\begin{multline*}
  \mathbb{E}\, \e^{-\gamma (S\wedge \tau_{n}-t)} V(z^{\widehat{\psi},\widehat{C}}(S\wedge \tau_{n}),Y(S\wedge \tau_{n}),S\wedge \tau_{n})\\
  = V(z,y,t)-\frac{1}{\alpha}\mathbb{E}\, a\int_{t}^{S\wedge \tau_{n}} \e^{-\gamma (s-t)}\left(\widehat{C}(s)z^{\widehat{\psi},\widehat{C}}(s)\right)^{\alpha} \d s, 
\end{multline*}
where $(\tau_{n}, n \in \mathbb{N})$  is a localizing sequence of stopping times and $S<T$ is a positive constant.

Condition \eqref{uniform} provides justification to passing to the limit under the expectation sign. Eventually, we arrive at
\[
 V(z,y,t) =  \frac{1}{\alpha} \mathbb{E}\left[a\int_{t}^{T}\e^{-\gamma (s-t)}\left(\widehat{C}(s)z^{\widehat{\psi},\widehat{C}}(s)\right)^{\alpha} \d s+b\e^{-\gamma (T-t)} (z^{\widehat{\psi},\widehat{C}}(T))^{\alpha}\right].
\]

Next, we need to show that the value function $V$ dominates the value for other admissible strategies. Repeating the procedure for any admissible strategy $(\psi,C)$, we get
\begin{multline*}
  \mathbb{E}\, \e^{-\gamma (S\wedge \tau_{n}-t)} V(z^{\psi,C}(S\wedge \tau_{n}),Y(S\wedge \tau_{n}),(S\wedge \tau_{n}))\\
  \leq V(z,y,t)-\frac{1}{\alpha}\mathbb{E}\, a\int_{t}^{S\wedge \tau_{n}} \e^{-\gamma (s-t)}\left(C(s)z^{\psi,C}(s)\right)^{\alpha} \d s.
\end{multline*}
By the positivity of $\alpha$, we can use the Fatou lemma to obtain
\begin{equation} \label{Fatou_ending}
\begin{aligned}
&V(z,y,t)\\
&\quad  \geq  \frac{1}{\alpha}\mathbb{E} \left[a\int_{t}^{T}\e^{-\gamma (s-t)}\left(C(s)z^{\psi,C}(s)\right)^{\alpha} \d s + b \e^{-\gamma (T-t)} (z^{\psi,C}(T))^{\alpha}\right].
\end{aligned}
\end{equation}
\qquad $\square$
\begin{remark}\label{R}
If  $\alpha <0$, then  the first part of the theorem holds true.  We do not know how to show that the value function $V$ dominates the value for other admissible strategies. The main problem is  to show the convergence of the term
\[
\mathbb{E}\, \e^{-\gamma (S\wedge \tau_{n}-t)} V(z^{\psi,C}(S\wedge \tau_{n}),Y(S\wedge \tau_{n}),S\wedge \tau_{n}).
\]
 \end{remark}

\begin{remark} Assume that we additionally have the possibility to allocate our resources in the stock market $S(t)\in \mathbb{R}^N$, with the dynamics
\[
\d S(t)=\textrm{diag}(S(t))\left[\overline r(t) +  \Sigma_{S}(t, Y(t))  \lambda(t,Y(t))]\d t +  \Sigma_{S}(t, Y(t)) \d W(t)\right].
\]
Then an  investment policy is  a pair $(\psi,\pi)$, where  $\psi$ is a signed measure on $[0,T^*]$ and $\pi\in \mathbb{R}^N$ are such that 
$$
\int_0^{T^*} \psi(\d x)+ \sum_{j=1}^N \pi_j=1. 
$$
The optimal investment policy  $(\widehat \psi_t, \widehat \pi_t)$, $t\in [0,T]$, should solve the system of equations
\begin{align*}
A(\widehat \psi_t)(t,Y(t)) + \Sigma_{S}^\top (t,Y(t))\widehat \pi_t=\frac {\lambda (t,Y(t))} {(1-\alpha)}  + \Sigma^{\top}(t,Y(t)) \frac{D_{y} G(t,Y(t))}{G(t,Y(t))}.  
\end{align*}
The existence of a sequence $\mathbf{x}= (x_{1},x_{2},\ldots,x_{l})$ such that the matrix 
$$
\left[ 
{\begin{array}{c}
\mathcal{F}(t,y, \mathbf{x}),    \Sigma_{S}^\top(t,y)  \\
 e\\
\end{array} }
\right] 
$$ 
has  rank $m+1$ ensures the existence of an optimal control $(\widehat \psi,\widehat \pi)$ with $\widehat \psi_t$ being a point measure. We are aware that the processes $S$ and $Y$ share the same Wiener process and this not cover full generality, but the analysis with another independent Wiener in the dynamics of $S$ or $Y$ is out of the scope of this paper. Partial results for the general problem can be
found for example in Hata et al. \cite{Hata} and Zawisza \cite{Zawisza3}. But this formulation is sufficient for example to cover the bond--stock mix problem of Brennan and Xia \cite{Brennan}. 
\end{remark}

\section{Examples}\label{S6}
Our assumptions   allow us  to consider the following important models with practical implementations. Detailed calculations of particular affine examples of the models presented below are given in the next section.
\begin{example}\textbf{(Consistent HJM models)}
We assume that the forward rate in the Musiela parametrization is given by $r(t)(x)=\phi(x,Y(t))$, where  $\phi \in C^{1,2}([0,T^*]\times \mathbb{R}^{n})$, and $Y$ is given by \eqref{E31}.  Note that for the short rate we have $\overline r(t)=\phi(0,Y(t))$. Then, see Filipovi\'{c} \cite[Proposition 9.1]{Filipovic},  in the case of $B$ and $\Sigma$ independent of $t$,  under Hypothesis  \ref{A1} and \ref{A2},   the consistency condition in Hypothesis   \ref{A3} holds if and only if  $\Phi(x,y)=\int_0^x \phi(u,y)\d u$ satisfies
\begin{align*}
\frac{\partial \Phi}{\partial x}(x,y)  = &\phi(0,y) + \langle B(y), D_{y} \Phi(x,y) \rangle \\ & \quad + \frac{1}{2} \sum_{i,j=1}^n \left(\Sigma(y) \Sigma^{\top}(y)\right)_{i,j}\left[  \frac{\partial ^2 \Phi}{\partial y_i\partial y_j} (x,y) - \frac{\partial \Phi}{\partial y_i}(x,y)\frac{\partial \Phi}{\partial y_j}(x,y)\right]. 
\end{align*}


\end{example}
\begin{example} \textbf{(Short term interest rate models)} 
Assume that 
$$
\d \overline r(t)=B(t,\overline r(t))\d t + \Sigma(t, \overline r(t))\left[\lambda (t)\d t+\d W(t)\right],
$$
where $B$ and $\Sigma$ are functions satisfying Hypothesis \ref{A1} and \ref{A2} and $\lambda$ is a deterministic bounded measurable function.  Then, in the framework of the HJM factor model, we take  $Y=\overline r$.   
\end{example}
\begin{example}\textbf{(Gaussian--separable HJM model (Ritchken--Sankara\-subramanian model)} Assume that the volatility has the form 
$\sigma(t)(x)=\beta(t) \nu(x+t)$, where $\beta$ and $\nu$ are deterministic functions. In this case, after assuming some mild  regularity of the function $\nu$ and $f(0,t)$, we have 
$$
\begin{aligned}
\d \overline r(t)&= \left[\frac{\partial f(0,t)}{\partial t}  -f(0,t) \frac{\nu'(t)}{\nu(t)}+ \int_{0}^{t} \beta(u) \nu(u) \d u  +\overline r(t)  \frac{\nu'(t)}{\nu(t)} \right] \d t\\
&\qquad + \beta(t) \nu(t)\d W^{*}(t).
\end{aligned}
$$
Thus the problem can be reduced simply to the short rate models. This specification allows for further generalization to the so-called 
the Cheyette models
\[
\sigma(t)(x)=\sum_{i=1}^{N}\beta_{i}(t) \frac{\nu_{i}(t)}{\nu_{i}(x+t)}.
\]
For more details we refer to Beyna \cite{Beyna}.
\end{example}
\begin{example} \textbf{(Quasi Gaussian HJM model)} Let  $$\sigma(t)(x)=\beta(t,\xi(t)) \nu(x+t),$$ where $\xi(t)$ is a diffusion and $\nu$ a deterministic function.  Then
\begin{align*}
\d \overline r(t)  &= \left[\frac{\partial f(0,t)}{\partial t}  -f(0,t) \frac{\nu'(t)}{\nu(t)}+ \int_{0}^{t} \beta(u,\xi(u)) \nu(u) \d u  +\overline r(t)  \frac{\nu'(t)}{\nu(t)} \right] \d t \\ &\qquad + \beta(t,\xi(t)) \nu(t)\d W^{*}(t).
\end{align*}
Clearly, for $\psi(r,\xi)= r$ and $Y=(\overline r,\xi)$ we  have $\overline{r}(t)= \psi(t,Y(t))$. Unfortunately, the strong ellipticity condition cannot be satisfied. In order to have a non-degenerate diffusion one can replace the term $\int_0^t \beta(u,\xi(u)) \nu(u) \d u\d t$ by its $\varepsilon$--perturbation
$$
\int_0^t \beta(u, \xi (u)) \nu(u) \d u\d t+ \varepsilon \d \tilde W(t),
$$
where $\tilde W$ is an independent Wiener process. 

A representative example of such a model is the so-called Cheyette HJM model with the function $\beta$ having the following affine structure 
$\beta(t,\xi):=\zeta_{1}(t)+\zeta_{2}(t)\xi$,
where $\zeta_{1}$ and $\zeta_{2}$ are deterministic functions. Note that in affine models (for the definition see the next section) taking $\varepsilon \to 0$ is an instantaneous operation and does not need separate justification. For more information about such models we refer to Pirjol and Zhu \cite{Pirjol}.
\end{example}

\section{Affine factor models}\label{S7}
Our aim here is to present particular examples of models which admit explicit solutions. As is Section \ref{S4} we assume that $\overline r(t)=\varphi(t,Y(t))$, where $Y$ is given by \eqref{E31}. We focus on affine models;  i.e. models with the following specifications
\begin{equation} \label{affine_spec}
\begin{aligned}
B(t,y)&:=B_{1}(t) + \langle B_{2}(t), y \rangle, \quad \Sigma(t,y):=\Sigma(t), \\ \lambda(t,y)&:=\lambda(t), \quad \varphi(t,y)=\varphi_{1}(t) + \langle \varphi_{2}(t),y\rangle,
\end{aligned}
\end{equation}
where $B_{1}(t)$, $B_{2}(t)$, $\Sigma(t)$, $\lambda(t)$, $\varphi_{1}(t)$,  $\varphi_{2}(t)$ are deterministic matrix-valued mappings. At the end of the present section we will consider  the CIR model, which is not in fact an affine factor model.  

\begin{proposition} Under the affine specification \eqref{affine_spec} the formula for the optimal pair $(\widehat{\psi},\widehat{C})$ reads
\begin{multline*}
\Sigma^\top(t)\int_{0}^{T^{*}}  \left[ \int_{t}^{t+x}P_{t,t+x} P_{k,t+x}^{-1}  \varphi_{2}(k)\d k\right] \psi_{t}(\d x)\\
=   \frac{ \lambda(t)}{1-\alpha}+   \Sigma^\top(t) \frac{D_yG(t,Y(t))}{G(t,Y(t)} ,
\end{multline*}
$C(t)= G(t,Y(t))^{-1}$, where 
\begin{equation} \label{affine1}
\begin{aligned}
G(t,y)&:=a^{\frac{1}{1-\alpha}}\int_{t}^{T} \eta(t,u,y) \d u + b^{\frac{1}{1-\alpha}} \eta(t,T,y),  \\ \eta(t,u,y)&:=  \e^{ m_{1,t,u} +\langle m_{2,t,u}, y\rangle +\frac{1}{2} \sigma_{t,u}^{2}},
\end{aligned}
\end{equation}
\begin{equation} \label{affine2}
\begin{aligned}
m_{2,t,u}&:=\left[ \frac{\alpha}{1-\alpha} \right] \left[\int_{t}^{u} P_{t,u} P_{k,u}^{-1} \varphi_{2}(k) \d k\right], \\ m_{1,t,u} + \frac{1}{2} \sigma_{t,u}^{2}  &=\int_{t}^{u}f(k,u)\d k,
\end{aligned}
\end{equation}
\begin{equation} \label{affine3}
\begin{aligned}
&f(t,u):=\frac{1}{2} \left\langle m_{2,t,u}, \Sigma(t) \Sigma^{\top}(t)  m_{2,t,u} \right\rangle + \langle B_{1}(t),  m_{2,t,u}\rangle \\
&\quad + \frac{1}{1-\alpha} \left[ \langle  m_{2,t,u},  \Sigma(t) \lambda(t) \rangle+  \frac{\alpha}{2(1-\alpha)} \|\lambda(t)\|^{2}  -\gamma \right] 
\end{aligned}
\end{equation}
and $P_{t,k}$ denotes the time-ordered path exponential of the matrix $B(t)$.
\end{proposition}
\begin{proof}
Note that in the affine framework $g$ given by \eqref{E42} is equal to 
\begin{align*}
g(t,y)&= \frac{1}{1-\alpha} \left[ \alpha \varphi_1(t) +\frac{\alpha}{2(\alpha -1)}\|\lambda (t)\|^2 -\gamma\right] + \frac{\alpha}{1-\alpha} \langle \varphi_2 (t),y\rangle \\
&=: g_0(t) + \langle g_1(t), y\rangle. 
\end{align*}
Moreover, the  process $\widetilde Y$ appearing in the Feynman--Kac representation \eqref{E44} and \eqref{E45} of  the function $G$ is Gaussian. Hence, for any $t\le u$,  the random variable  $\int_{t}^{u}  g(k,\widetilde Y(k))\d k$ has a Gaussian distribution $\mathcal{N}(m_{1,t,u} + \langle m_{2,t,s}, y\rangle,\sigma_{t,u})$, and consequently we have \eqref{affine1}.

On the other hand, substituting the above  function $G$   to equation \eqref{E43}, we infer  that the function $\eta$ satisfies
\begin{multline*} 
-\frac{\partial \eta}{\partial t}  + \frac{1}{2} \operatorname{Tr}\left( \Sigma \Sigma^{\top}D^{2}_{yy} \eta\right)+ \langle B_{1} + B_{2} y,  D_{y}\eta\rangle   \\
+\frac{1}{(1-\alpha)}   \langle \Sigma \lambda,   D_{y}\eta\rangle  +g_0 + \langle g_1, y\rangle=0,
\end{multline*}
$\eta(t,t,y)=1$.  So consequently 
\[
\frac{\partial }{\partial t}m_{2,t,u} =B_{2}(t) m_{2,t,u} + \frac{\alpha}{1-\alpha} \varphi_2 (t), \quad m_{2,t,t}=0,
\]
and therefore
\[
m_{2,t,u}=\left[ \frac{\alpha}{1-\alpha} \right] \left[\int_{t}^{u} P_{t,u} P_{k,u}^{-1} \varphi_{2}(k) \d k\right], 
\]
where $P_{t,k}$ denotes the time-ordered path exponential of the matrix $B(t)$. For the function $m_{1,t,s} + \frac{1}{2} \sigma_{t,s}^{2}$ we have 
$$
\frac{\partial }{\partial t} \left[m_{1,t,u} + \frac{1}{2} \sigma_{t,u}^{2} \right] =f(t,u), \qquad  m_{1,t,t} + \frac{1}{2} \sigma_{t,t}^{2} =0, 
$$
where $f(t,u)$ equals 
\begin{multline*}
\frac{1}{2} \left\langle m_{2,t,u}, \Sigma(t) \Sigma^{\top}(t)  m_{2,t,u} \right\rangle
+ \langle B_{1}(t),  m_{2,t,u}\rangle \\+ \frac{1}{1-\alpha} \left[ \langle  m_{2,t,u},  \Sigma(t) \lambda(t) \rangle+  \frac{\alpha}{2(1-\alpha)} \|\lambda(t)\|^{2}  -\gamma \right]. 
\end{multline*}
Hence 
$$
m_{1,t,u} + \frac{1}{2} \sigma_{t,u}^{2}  =\int_{t}^{u}f(k,u)\d k.
$$

\end{proof}

\begin{remark}
It is worth to stress that in the investment--consumption model under the affine factor assumption the optimal portfolio weights are not linear combination of factors. This is in contradiction with the pure investment problem (see Bielecki and Pliska \cite{Bielecki2}).
\end{remark}

\subsection{Short rate affine models}\label{S71}
Now, let us consider more particular examples. The model is  \emph{short rate affine}  if the price $P(t,S)$ at time $t$ of the  zero coupon bond with maturity $T$ is  
\begin{equation}\label{E71}
P(t,T)= \e^{m(T-t)-n(T-t)\overline r (t)},
\end{equation}
where $m$ and $n$ are deterministic functions. Moreover it is assumed that the short rate  $\overline r $ is a diffusion process 
\[
\d \overline r (t)= B(\overline r (t))\d t + \Sigma (\overline r (t))\left[ \lambda(\overline r(t))\d t + \d W(t)\right]. 
\]
In the notation of Section \ref{S4}, $Y=\overline r$. We assume that $B$,  $\Sigma$ and $\lambda$ satisfy Hypotheses \ref{A1} to \ref{A3}. Barski and Zabczyk \cite{Barski-Zabczyk2} have proved that in the affine models of bond prices are local martingales wrt. martingale measures if 
and only if $B(r)=a+br$ and ($\Sigma(r)=c\sqrt{r}$ or $\Sigma(r)=c$), where $a,b,c \in \mathbb{R}$ are constants. Obviously, \eqref{E71} requires additional assumptions on $\lambda$. Under the Musiela parametrization, see Appendix \ref{Appendix},  we have 
\[
P(t)(x)= \e^{m(x)-n(x)\overline r (t)}= F(t,t+x,\overline r(t)),
\quad 
F(t,t+x,r)=\e^{m(x)-n(x)r}, 
\]
and 
$$
r(t)(x)=-m'(x)+n'(x)\overline r(t).
$$

Since 
$$
\frac{\partial }{\partial r} \log F(t,t+x,r)=-n(x)
$$
the formula for the optimal investment strategy reads 
\begin{equation} \label{formula_affine}
\int_0^{T^*} n(x) \widehat  \psi_t(\d x) = \frac{1}{(\alpha -1)}  \frac{\lambda(\overline r(t))}{\Sigma (\overline r(t))}  -  \frac{D_rG(t,\overline r(t))}{G(t,\overline r(t))}.
\end{equation}
The optimal consumption rate is given by
\begin{equation} \label{cons_affine}
\widehat C(t)= G(t,\overline r(t))^{-1}.
\end{equation} 
\subsubsection{Vasicek model}\label{SVasicek}
First we will consider the so called \emph{Vasicek model}
\begin{equation} \label{Vasicek_model}
\d \overline r(t)= \left[\beta- \kappa \overline r(t)\right]\d t+ \sigma  \left[ \lambda \d t+ \d W(t)\right] ,
\end{equation}
where $b$, $\kappa$, $\lambda$, and $\sigma>0$ are constant.
Note, that under Vasicek model we have 
$$
n(x)= \frac{1-\e^{-\kappa x}}{\kappa},\quad 
m(x)= -\beta\int_0^x n(y)\d y +\frac{\sigma^2}{2} \int_0^x \left(n(u)\right)^2\d u,
$$
and 
$$
 \frac{\d U(t)(x)}{U(t)(x)}= \overline r(t)\d t- \sigma n(x)\left[ \lambda \d t + \d W(t)\right]. 
$$
\begin{proposition} \label{PVasicek} In the Vasicek model \eqref{Vasicek_model} the optimal pair $(\widehat{\psi},\widehat{C})$ is given by 
\eqref{formula_affine}--\eqref{cons_affine} where
\begin{equation}\label{E72}
G(t,r)=a^{\frac{\alpha}{1-\alpha}} \int_{t}^{T} \e^{ m_{1,t,u} +m_{2,t,u} r  +\frac{1}{2} \sigma_{t,u}^{2}} \d u  + b^{\frac{\alpha}{1-\alpha}} \e^{ m_{1,t,T} +m_{2,t,T} r +\frac{1}{2} \sigma_{t,T}^{2}},
\end{equation}
and
\begin{equation*}
\begin{aligned}
\sigma_{t,u}^{2}&:=\overline{\alpha}^2 \int_{t}^{u}n^2(k)\sigma^2 \d k ,\\
 m_{1,t,u}&:=\overline{\alpha}\int_{t}^{u}\left[n(k)\left( \beta +  \frac{1}{(1-\alpha)}   \lambda \sigma  \right) + \frac{1}{2(1-\alpha)} \lambda^2  -\gamma \right]\d k, \\
 m_{2,t,u}&:=\overline{\alpha} \,n(u-t),
\end{aligned}
\end{equation*}
$\overline {\alpha}:= \alpha/(1-\alpha)$. 
\end{proposition}
\begin{proof}
Note that equation \eqref{E43} for $G$ has now the following form 
$$
\begin{aligned}
\frac{\partial G}{\partial t}(t,r)+ \tilde L G(t,r) + g(r) G(t,r)  +a^{\frac{1}{1-\alpha }} &=0,\\
G(T,y)&= b^{\frac{1}{1-\alpha}}, 
\end{aligned}
$$
where 
$$
\tilde LG(t,r)= \frac{\sigma^2}{2} \frac{\partial^2 G}{\partial r^2}(t,r)+ \left[  \beta-\kappa r +\frac{\alpha}{(1-\alpha)}  \sigma \lambda\right] \frac{\partial G}{\partial r}(t,r) 
$$
and 
$$
g(r)= \frac{1}{1-\alpha}\left[\alpha r   + \frac{\alpha}{2(1-\alpha)}  \lambda^2   -\gamma\right]. 
$$
Thus 
$$
G(t,r) =\mathbb{E}\left\{ a^{\frac{1}{1-\alpha}}  \int_{t}^{T}\e^{  \int_{t}^{u}  g(\widetilde{r}(k))\d k}\d u + b^{\frac{1}{1-\alpha}}\e^{ \int_{t}^{T}  g(\tilde r(k))\d k} \right\},
$$
where
\[
\d \widetilde{r}(k)= \left[\beta- \kappa \widetilde{r}(k) + \frac{1}{(1-\alpha)}   \lambda \sigma   \right] \d k + \sigma \d W(k), \quad \widetilde{r}(t)=r. 
\]

Note, that
\[
\int_{t}^{u} \widetilde{r}(k)\d k= n(u-t)r + \int_{t}^{u}n(k)\left[ \beta +  \frac{1}{(1-\alpha)}   \lambda \sigma  \right]\d k + \int_{t}^{u}n(k)\sigma \d W(k).
\]
Therefore, the integral $\int_{t}^{u} \widetilde{r}(k)\d k$ is normally distributed, and in this particular case $G$ is given by \eqref{E72}. 
\end{proof}
Note that the optimal investment is determined by the condition
\[
\int_0^{T^*} \frac{1-\e^{-\kappa x}}{\kappa} \widehat  \psi_t(\d x) =  \frac{\lambda}{(\alpha -1)\sigma}  -  \frac{D_rG(t,\overline r(t))}{G(t,\overline r(t))},\quad \int_0^{T^*}  \widehat  \psi_t(\d x)=1. 
\]
In particular, we can take $\widehat\psi_t(\d x)= \widehat \eta_0(t)\delta_0(\d x)+ \widehat \eta_{\overline x}(t)\delta _{\overline x}(\d x)$, where $\overline x\in (0,T^*]$ is an arbitrary fixed time to maturity. The process $(\widehat \eta_0(t),\widehat\eta_{\overline x}(t))$ is determined by 
$$
\widehat  \eta_{\overline x}(t)=  \frac{\kappa }{(1-\e^{-\kappa \overline {x}})} \left[ \frac{\lambda}{(\alpha -1)\sigma}  -  \frac{D_rG(t,\overline r(t))}{G(t,\overline r(t))}\right],\quad \widehat \eta _0(t)=1-\widehat\eta_{\overline x}(t). 
$$

It should be noted that the solution for the simpler Merton model 
\[
\d \overline  r(t)= \beta \d t+ \sigma    \left(\lambda \d t+ \d W(t)\right)
\]
can be derived  by passing in $\kappa \to 0$ in the Vasicek model. Thus, in the Merton model, $n(x)=x$ and $m(x)=-\beta x^2 + (\sigma^2/6) x^3$. In particular, the condition for the optimal investment has the form 
$$
\int_0^{T^*}x  \widehat \psi_t(\d x)= \frac{\lambda}{(\alpha -1)\sigma}  -  \frac{D_rG(t,\overline r(t))}{G(t,\overline r(t))},\quad \int_0^{T^*}  \widehat  \psi_t(\d x)=1. 
$$

\subsubsection{CIR model}\label{SCIR}
The second important model worth to consider is the Cox--Ingersoll--Ross model. It does not satisfy our Hypotheses \ref{A1} and  \ref{A2}. However, we  are able  to perform some explicit calculations. In fact we will derive an exact formula for the solution  $G$ to equation \eqref{E43}. In this way we obtain a candidate for the value function for the corresponding control problem.    We perform calculations only for the case $\alpha <0$,  see however Remark \ref{R}. For other parameters it might happen that the value function has infinite value  (see for example Korn and Kraft \cite[Proposition 3.2]{Kraft}). We start by introducing elementary facts about the CIR model
\begin{align*} 
\d \overline r(t)&= (\beta-\kappa \overline  r(t))\d t+ \sigma\sqrt{\overline r(t)} \left(\lambda \d t + \d W(t)\right),\\
n(x)&= \frac{\sinh \rho x}{\rho\cosh \rho x +\frac{\kappa}{2} \sinh \rho x}, \\
m(x)&= \frac{2\beta}{\sigma}\log \left( \frac{\e^{\frac{\kappa x}{2}}}{\rho \cosh \rho x +\frac{\kappa}{2} \sinh \rho x}\right), 
\\
\rho &= \frac 12 \left(\kappa ^2 +2\sigma^2\right)^{1/2}, \\
 \frac{\d U(t)(x)}{U(t)(x)}&=  \overline r(t)\d t -\sigma \sqrt{\overline r(t)} n(x) \left(\lambda \d t +  \d W(t)\right). 
 \end{align*}
Usually it is assumed that $2\beta\ge \sigma^2$. To obtain closed form solution we assume here $\lambda(r):=\lambda \sqrt{r}$.

Note that we have
\[
\d \widetilde{r}(k)= \left[\beta- \widetilde{\kappa} \widetilde{r}(k)  \right] \d k + \sigma \sqrt{\widetilde{r}(k) } \d W(k),\quad \widetilde {r}(t)=r, \quad \widetilde{\kappa}=\kappa - \frac{\lambda \sigma}{(1-\alpha)}  
\]
and 
\[
\mathbb{E}\,\e^{\int_{t}^{u}\frac{\alpha}{1-\alpha}\left[1+\frac{ \lambda^2}{2(1-\alpha)}\right] \widetilde{r}(k) \d k}= \e^{\widetilde{m}(u-t) - \widetilde{n}(u-t) r},
\]
where
\begin{align*}
\widetilde{n}(x)&=\frac{|\alpha|}{1-\alpha}\left[1+\frac{ \lambda^2}{2(1-\alpha)}\right]  \frac{\sinh \widetilde{\gamma} x}{\widetilde{\gamma} \cosh \widetilde{\gamma} x +\frac{\widetilde{\kappa}}{2} \sinh \widetilde{\gamma} x},\\
\widetilde{m}(x)&= \sqrt{\frac{|\alpha|}{(1-\alpha)}\left[1+\frac{ \lambda^2}{2(1-\alpha)}\right] }  \frac{2\beta}{\sigma}\log \left( \frac{\e^{\frac{\widetilde{\kappa} x}{2}}}{\widetilde{\gamma} \cosh \widetilde{\gamma} x +\frac{\widetilde{\kappa}}{2} \sinh \widetilde{\gamma} x}\right),
\\
\widetilde{\gamma} &= \frac 12 \left(\widetilde{\kappa}^2 +2\frac{|\alpha|}{(1-\alpha)}\sigma^2\right)^{1/2}.\\
\end{align*}

So the solution  to equation \eqref{E43} corresponding to CIR model  is given by  
\begin{equation}\label{E74}
G(t,r)=a^{\frac{\alpha}{1-\alpha}} \int_{t}^{T} \e^{ m_{1,t,u} +m_{2,t,u} r} \d u + b^{\frac{\alpha}{1-\alpha}} \e^{ m_{1,t,T} +m_{2,t,T} r },
\end{equation}
where
$$
 m_{1,t,u}:=\widetilde{m}(u-t)  - \frac{\gamma}{1-\alpha}  (u-t), \qquad  m_{2,t,u}:=\widetilde{n}(u-t). 
$$
Summing up, we get
 
\begin{proposition} In the CIR model a candidate for the optimal pair $(\widehat{\psi},\widehat{C})$ is given by 
\begin{equation*} 
\int_0^{T^*} \widetilde{n}(x) \widehat  \psi_t(\d x) = \frac{1}{(\alpha -1)}  \frac{\lambda(\overline r(t))}{\Sigma (\overline r(t))}  -  \frac{D_rG(t,\overline r(t))}{G(t,\overline r(t))}, \quad \widehat C(t)= G(t,\overline r(t))^{-1},
\end{equation*}
 where $G(t,r)$ is given by \eqref{E74}.
\end{proposition}

\subsubsection{Multidimensional model}
It is right time to present a certain multidimensional model. Here we focus on  the \emph{G2++} model, which is important from the perspective of applications. Let 
$\overline r(t) = Y_{1}(t)+Y_{2}(t)$,
where, 
\begin{align*} 
\d Y_{1}(t)&=-\kappa_{1}Y_{1}(t) \d t + \sigma_{1}\left[ \lambda_1\d t + \d W_1(t)\right],  \\
\d Y_{2}(t)&=-\kappa_{2}Y_{2}(t) \d t + \sigma_{2}\left[ \rho \left[ \lambda_1 \d t + \d W_1(t)\right] + \sqrt{1-\rho^{2}}\left[ \lambda_2 \d t + \d W_2(t)\right]\right], \label{g23}
\end{align*}
 and  $W_{1}$,$W_{2}$ are independent Brownian motions.  In other words, the short rate is a sum of two correlated Vasicek (or Ornstein--Uhlenbeck) processes. 

\begin{proposition} In the G2++ model  the optimal pair $(\widehat{\psi},\widehat{C})$ is given by
\[
-\int_0^{T^*} (n_{1}(x),n_{2}(x))  \widehat {\psi}_t(\d x) = \frac{(\lambda_1,\lambda_2)\Sigma^{-1}}{1-\alpha}  +  \frac {D_{y} G(t,Y(t))}{G(t,Y(t))}, 
\]
$\widehat C(t)= G(t,Y(t))^{-1}$,  where 
\begin{align*}
G(t,y)&=a^{\frac{1}{1-\alpha}}\int_{t}^{T} \eta(t,u,y) \d u + b^{\frac{1}{1-\alpha}} \eta(t,T,y),  \\ 
\eta(t,u,y)&=  \e^{ m_{1,t,u} +\langle m_{2,t,u}, y\rangle +\frac{1}{2} \sigma_{t,u}^{2}},\\
\sigma_{t,u}^{2}&:=\left[\frac{\alpha}{1-\alpha}\right]^{2}\left[\int_{t}^{u}(n_{1}(k)\sigma_{1} +n_{2}(k) \sigma_{2}\rho)^2 +n_{2}^{2}(k)\sigma_{2}^{2} (1-\rho^{2})  \right] \d k,\\
 m_{1,t,u}&:= \biggl[\frac{\alpha}{1-\alpha}\biggr] \int_{t}^{u} \biggl[n_{1}(k)\beta_1+ n_{2}(k)\beta_2 + \frac{\alpha}{2(1-\alpha)} \|\lambda\|^2 + \varphi  -\gamma \biggr]\d k, \\
 m_{2,t,u}&:=\left(\frac{\alpha}{1-\alpha} n_{1}(u-t),\frac{\alpha}{1-\alpha}n_{2}(u-t)\right),\\
\Sigma &:= \left[ {\begin{array}{cc}
   \sigma_{1} & 0\\
   \rho \sigma_{2} & \sigma_{2}\sqrt{1-\rho^{2}} \\
  \end{array} } \right],  
\end{align*}
and 
$$ 
n_{1}(x):= \frac{1-\e^{-\kappa_{1} x}}{\kappa_{1}},\quad 
n_{2}(x):= \frac{1-\e^{-\kappa_{2} x}}{\kappa_{2}}.
$$
\end{proposition}
\begin{proof}
Here the formula for the function $G$ looks as follows
\[
G(t,y) =\mathbb{E} \left\{  \int_{t}^{T}a^{\frac{1}{1-\alpha}}\e^{  \int_{t}^{u}  g(\widetilde{Y}_{1}(k) +\widetilde{Y}_{2}(k))\d k} \d u + b^{\frac{1}{1-\alpha}}\e^{  \int_{t}^{T}  g(\widetilde{Y}_{1}(k) +\widetilde{Y}_{2}(k))\d k}  \right\},
\]
where
$$
 g(z) = \frac{1}{1-\alpha}\left[\alpha z + \frac{\alpha}{2(1-\alpha)} \| (\lambda_1,\lambda_2)\|^2  -\gamma\right]
$$
and 
\begin{align*}
\d \widetilde{Y}_{1}(u)&=\left[ \beta_{1} -\kappa_{1}\widetilde{Y}_{1}(u) \right]\d u + \sigma_{1}\d W_{1}(u), \\
\widetilde{Y}_{1}(t)&=y_1,\\
\d \widetilde{Y}_{2}(u)&=\left[ \beta_2-\kappa_{2}\widetilde{Y}_{2}(u) \right] \d t +\sigma_{2}(\rho  \d W_{1}(u) + \sqrt{1-\rho^{2}}\d W_{2}(u)),\\
\widetilde{Y}_{2}(t)&=y_2. 
 \end{align*}
$$
\beta_1:= \frac{1}{(1-\alpha)}   \lambda_{1} \sigma_{1}, \qquad \beta_2:= \frac{1}{(1-\alpha)} \left( \lambda_{1} \sigma_{2}\rho + \lambda_{2} \sigma_{2} \sqrt{1-\rho^{2}}\right). 
$$

Taking the advantage of  the Vasicek model  we obtain 
\begin{multline*}
\int_{t}^{u} (\widetilde{Y}_{1}(k) +\widetilde{Y}_{2}(k))\d k = n_{1}(u-t)y_{1} +  n_{2}(u-t)y_{2} +  \int_{t}^{u}n_{1}(k)\beta_1\d k \\  + \int_{t}^{u}n_{2}(k)\beta_2\d k + \int_{t}^{u}(n_{1}(k)\sigma_{1} +n_{2}(k) \sigma_{2}\rho) \d W_{1}(k) \\+   \int_{t}^{u}n_{2}(k)\sigma_{2} \sqrt{1-\rho^{2}} \d W_{2}(k).
\end{multline*}
This implies the desired identities.
\end{proof}

In particular, if we take $$\widehat\psi_t(\d x)= \widehat \eta_0(t)\delta_0(\d x)+ \widehat \eta_{ x_{1}}(t)\delta _{x_{1} }(\d x)+\widehat \eta_{ x_{2}}(t)\delta _{x_{2} }(\d x),$$ where \\ $ x_{1}, x_{2} \in (0,T^*]$, then the process $(\widehat \eta_0(t),\widehat \eta_{ x_{1}}(t)),\widehat \eta_{ x_{2}}(t))$ is determined by 
\begin{align*}
(\widehat \eta_{ x_{1}}(t),\widehat \eta_{ x_{2}}(t))&=  \frac{(\lambda_1,\lambda_2)\Sigma^{-1}M(x_{1},x_{2})^{-1}}{\alpha-1}  -  \frac {D_{y} G(t,Y(t))M(x_{1},x_{2})^{-1}}{G(t,Y(t))},\\ \widehat \eta _0(t)&=1-\widehat \eta_{x_{1}}(t)-\widehat \eta_{x_{2}}(t), \notag
\end{align*}
where
\[
M(x_{1},x_{2}):=\left[ {\begin{array}{cc}
   n_{1}(x_{1}) & n_{1}(x_{2})\\
   n_{2}(x_{1}) & n_{2}(x_{2}) \\
  \end{array} } \right].
\]

\section{Infinite horizon problem}\label{S8}

Recall that in the infinite horizon case the reward functional is given by \eqref{E38}, the short rate has  the form $\overline r(t)= \varphi(Y(t))$ where $Y$ solves \eqref{E31}.  We assume that Hypothesis \ref{A1} to \ref{A3} are fulfilled. Moreover, in this section  $\varphi$ as well as the coefficients $B, \Sigma$ and $\lambda$  appearing in equation \eqref{E31} do not depend on time variable.  

Let 
$$
V(z,y)=\sup_{\psi, C} J(z,y,\psi,C)
$$
be the value function of the investor. 

We start by providing a simple example of the model that satisfies our  Hypotheses  \ref{A1} to \ref{A3},  but produces infinite value function $V$. 
\begin{example}\label{Ex81}
Consider the  Merton model 
\[
\d \overline r(t) = \beta  \d t + \sigma \left( \lambda \d t + \d W(t)\right).
\]
Choose $\psi=\delta_{0}$, which corresponds  to the investment in the bank account only,  and fix  the consumption at a constant level $c> 0$. Then  $C(t)\equiv c$ and 
\[
\d z^{\psi,C}(t)=\left[\overline r(t) - c\right]z^{\psi,C}(t) \d t.
\]
Thus 
$$
z^{\psi,C}(t)= z\e^{\int_0^t \left[ \overline r(u)-c\right]\d u},
$$
and 
\[
J(z,r,\psi,C) = \frac{cz}{\alpha} \, \mathbb{E} \int_{0}^{+\infty} \e^{\int_{0}^{t} \left[\alpha \overline r(u)-\alpha c-\gamma\right] \d u}\d t. 
\]
Note that 
\[
\int_0^t \overline{r}(u)\d u = \frac{(\beta +\lambda) t^2}{2}  + \sigma \int_0^t W(u)\d u. 
\]
Since $\int_{0}^{t} W(u) \d u$ has the normal distribution with variance  $\frac{t^{3}}{3}$ we have   $J(z,r,\psi,C)=+\infty$, see Synowiec \cite{Synowiec}.
\end{example}

The HJB approach gives the following candidate for the value function $V$, optimal consumption $\widehat C$ and investment strategy $\widehat \psi$. Namely, we can expect that 
\begin{align}\label{E81}
V(z,y)&= G(y)^{\frac{1}{1-\alpha}}z^\alpha, \\
\widehat C(t)&= G(Y(t))^{-1}, \label{E82}
\end{align}  
and  $\widehat\psi\in \mathcal{M}_{+\infty}$ is  such that  
\begin{equation} \label{E83}
A(\widehat{\psi}_t)=   \frac{\lambda(Y(t))}{1-\alpha}  +   \Sigma^{\top} (Y(t))\frac{D_y G(Y(t))}{G(Y(t))}, 
\end{equation}
where $G$ solves the following elliptic equation
\begin{equation} \label{E84}
LG(y)+\frac{\alpha}{(1-\alpha)}  \langle \Sigma (y)\lambda (y), D_yG(y)\rangle + g(y) G(y) + 1=0.  
\end{equation}
Recall that $L$ given by \eqref{E41} is the generator of the diffusion defined by \eqref{E31} and $g$ is given by \eqref{E42}.  Finally, we can expect that $G$ has the following Feynman--Kac representation
\begin{equation} \label{E85}
G(y) =\mathbb{E}  \, \int_{0}^{+\infty}\e^{ \int_{0}^{u}  g(\widetilde{Y}(k))\d k}\d u,
\end{equation}
where $\widetilde{Y}$ solves the stochastic differential equation 
$$
\d  \widetilde{Y}(t)= \left[B(\widetilde{Y}(t)) + \frac{1}{(1-\alpha)}   \Sigma( \widetilde{Y}(t))   \lambda( \widetilde{Y}(t)) \right] \d t+ \Sigma(\widetilde{Y}(t)) \d W(t)
$$
with  $\widetilde {Y}(0)=y$.  Taking into account Example \ref{Ex81}, we see that a special care have to be taken to ensure the existence of a solution to \eqref{E85} or to the convergence of the  integral in \eqref{E85}.  Moreover, at the end we will have to verify the assumption of the verification theorem. 

Below we present general theorem which ensures such convergence. Let $h\colon \mathbb{R}^n\mapsto \mathbb{R}$ be a Lipschitz continuous function. Let us consider the elliptic equation of the form
\begin{equation}\label{E86}
LG(y) + h(y)G+ 1 =0, \quad \quad y \in \mathbb{R}^{n}.
\end{equation}

\begin{theorem} \label{maininfini}
In addition to Hypotheses \ref{A1} to \ref{A3} suppose that there exists a constant $L_{2}>0$ and a function  $\kappa\colon [0,+\infty) \times \mathbb{N} \to \mathbb{R} $ decreasing in the first argument, such that
\[
\langle B(x)- B(y), x-y\rangle  + \frac{1}{2} \|\Sigma(x)- \Sigma(y)\|^2 \leq- L_{2}\|x-y\|^{2},
\]
and 
\begin{equation} \label{E87}
\mathbb{E} \, \e^{\int_{0}^{t} 2h(Y(u)) du}  \le \kappa(t,n),   \quad y \in B(0,n),  \quad \int_{0}^{+ \infty} \sqrt{t\kappa(t,n)} \d t <+\infty.
\end{equation}
Then
\[
G(y):=\mathbb{E} \int_{0}^{+\infty} \e^{\int_{0}^{t}h(Y(u)) \d u} \d t 
\]
is a classical, $C^2(\mathbb{R}^n)$,  solution to \eqref{E86}.

\end{theorem}

\begin{proof}
Consider the parabolic problem
\[
\frac{\partial  G}{\partial t}(t,y) -  LG(t,y) - h(y)G(t,y)- 1=0, \qquad G(0,y)=0. 
\]
We have 
\[
G(t,y)= \mathbb{E} \int_{0}^{t} \e^{\int_{0}^{s} h(Y(k))\d k} \d s, \qquad Y(0)=y.
\]
Note, that 
\[
\frac{\partial G}{\partial t}(t,y)=\mathbb{E}\, \e^{\int_{0}^{t}h(Y(k))\d k} \le  \left[\mathbb{E}\, \e^{\int_{0}^{t}2h(Y(k))\d k}\right]^{\frac{1}{2}}.
\]
Therefore, by \eqref{E87},   $\frac{\partial G}{\partial t}$  converges to 0 as $t \to +\infty$,   uniformly on each ball $B(0,n)$.

Secondly we need the estimate for the Lipschitz constant of the function $G$ in $y$. Let $Y(\cdot;y)$ be the solution to SDE \eqref{E31} with initial condition $Y(0;y)=y$. By  standard SDE estimates there exists a constant $M>0$ such that for all $t >0$ and $y_1,y_2\in \mathbb{R}^n$, 
\[
\mathbb{E} \, \| Y(t;y_{1}) -  Y(t;y_{2})\|^{2} \le M \e^{-2L_{2} t}\|y_{1}-y_{2}\|^{2}.
\]
Because the function $h$ is Lipschitz continuous there exists a constant $N>0$ such that
\begin{align*}
&|G(t,y_{1}) - G(t,y_{2})|\\ \quad &\leq N\,  \mathbb{E} \, \int_{0}^{t}  \e^{\max \left\{\int_{0}^{s} h(Y(k; y_{1})) \d k,\int_{0}^{s} h(Y(k;y_{2}))\d k \right\}} \int_{0}^{s} \| Y(k;y_{1}) -  Y(k;y_{2})\| \d k \d s \\ &\leq N \int_{0}^{t}\left[\mathbb{E}\int_0^s  \| Y(k;y_{1}) -  Y(k;y_{2})\| ^2  \d k\right]^{\frac{1}{2}} \\
&\quad \times \left[\mathbb{E}\, s \e^{2\max \left\{\int_{0}^{s} h(Y(k;y_{1})) \d k,\int_{0}^{s} h(Y(k;y_{2})) \d k\right\}} \right]^{\frac{1}{2}} \d s.
\end{align*}

Passing from $y_{1}$ to $y_{2}$ and using the dominated convergence theorem we get the estimate
\[
\|D_{y}G(t,y)\| \leq  N_1 \int_{0}^{t} \left[ \int_{0}^{s}\e^{-2L_{2} k} \d k  \right]^{\frac{1}{2}}\left[\mathbb{E}\, s \e^{\int_{0}^{s}2h(Y(k;y))\d k} \right]^{\frac{1}{2}} \d s.
\]
Finally, for $y \in B(0,n)$ we get
\[
\|D_{y}G(t,y)\| \leq \frac{N_1}{\sqrt{2L_{2}}} \int_{0}^{+\infty} \sqrt{s \kappa(s,n)}\d s.
\]

Almost the same estimates can be done for the Lipschitz constant of the derivative $\frac{\partial G}{\partial t}$. In fact  we have
\[
\left\vert \frac{\partial G}{\partial t}(t, y_{1}) - \frac{\partial G}{\partial t}(t,y_{2})\right\vert  = \left| \e^{\int_{0}^{t}2h(Y(k;y_{1}))\d k}- \e^{\int_{0}^{t}2h(Y(k;y_{2}))\d k}\right|. 
\]
And by repetitive arguments we arrive at the estimate 
\[
\left\vert \frac{\partial G}{\partial t}(t,y_{1}) - \frac{\partial G}{\partial t}(t,y_{2})\right \vert \leq \widetilde{N} \sqrt{\kappa(t,n)}\|y_{1}-y_{2}\|, \quad  y_{1},y_{2} \in B(0,n).
\]

Now we may use the Schauder estimates (see e.g. Gilbarg - Trudinger \cite[Theorem 6.2]{Gilbarg}) to prove that there exists a sequence $(t_{n}, n \in \mathbb{N})$ such that  $G(y,t_{n})$ converges  uniformly on each ball to the function $G \in C^{2}(\mathbb{R}^{n})$ satisfying   \eqref{E86}.

\end{proof}

\begin{remark}
Condition \eqref{E87} is not in an analytic form but it can be easily verified for example in the affine model framework.
\end{remark}

In the case of unbounded function $\varphi$ the use of the verification theorem  needs a  justification. To do this consider an arbitrary  sequence of finite time investment horizons $(t_{n}; n \in \mathbb{N})$ and the corresponding sequence of the value functions
\[
V_{t_n}(z,y):=\sup_{\psi,C} \mathbb{E}\, \frac{1}{\alpha} \int_0^{t_{n}}\e ^{-\gamma t} \left(C(t)z^{\psi,C}(t)\right)^\alpha \d t.
\]
Let $\left((\widehat{\psi}_{n},\widehat{C}_{n});n \in \mathbb{N} \right)$ be the corresponding sequence of optimal pair of controls.

\begin{theorem}\textbf{(Verification theorem)}\label{T84}  Assume Hypotheses \ref{A1} to \ref{A4}. Additionally assume that  there exists a sequence $(t_{n}; n \in \mathbb{N})$, $t_{n} \to \infty$ such that $V_{t_n}(z,y) \to V(z,y)$, where $V$ is given by \eqref{E81} and $G$ is a ${C}^{2} (\mathbb{R}^{n})$) classical solution to \eqref{E84}. Then any pair $(\widehat{\psi},\widehat{C})$ satisfying \eqref{E82}, \eqref{E83} and 
\begin{equation}\label{E88}
\begin{split}
&\mathbb{E} \sup_{0\leq k \leq t} \left[\left(z_{k}^{\widehat{\psi},\widehat{c}} \right)^{\alpha} G^{\frac{1}{1-\alpha}}(Y_{k})\right]<+\infty\qquad  \forall\, t\ge 0, \\ &\lim_{t \to \infty} \mathbb{E}\, \e^{-\gamma t} V(z^{\widehat{\psi},\widehat{C}}(t),Y(t)) =0,
\end{split}
\end{equation}
is an optimal solution for the infinite horizon optimization problem.
\end{theorem}
\begin{proof} Let  $V_{t_n}(z,y)$ converges  to $V(z,y)$. Choose any admissible strategy $(\psi,C)\in \mathcal{A}_{+\infty}$. Note that 
\[
V_{t_n}(z,y) = \mathbb{E} \frac{1}{\alpha} \int_{0}^{t_{n}}e^{-\gamma s}\left(\widehat{C}_{n}(s)z^{\widehat{\psi}_{n},\widehat{C}_{n}}(s)\right)^{\alpha} \d s \geq \mathbb{E} \, \frac{1}{\alpha} \int_{0}^{t_{n}}e^{-\gamma s}\left(C(s)z^{\psi,C}(s)\right)^{\alpha} \d s .
\]
 This ensures that
 \begin{equation} \label{E89}
 V(z,y)= \sup_{\psi,C} \mathbb{E}\, \frac{1}{\alpha}\int_{0}^{+\infty}e^{-\gamma s} \left(C(s)z^{\psi,C}(s)\right)^{\alpha} \d s.
 \end{equation}
 Now, we need only to prove that the supremum in \eqref{E89} is attained at $(\widehat{\psi},\widehat{C})$.  Let us apply  the It\^o formula to obtain the dynamics of  $\e^{-\gamma t} V(z^{\widehat{\psi},\widehat{C}}(t),Y(t))$.  We obtain
\begin{equation*}
  \mathbb{E}\, \e^{-\gamma t \wedge \tau_{n}} V(z^{\widehat{\psi},\widehat{C}}(t\wedge \tau_{n}),Y(t\wedge \tau_{n}))=V(z,y) -\mathbb{E}\frac{1}{\alpha}\int_{0}^{t \wedge \tau_{n}}e^{-\gamma s}\left(\widehat{C}(s)z^{\widehat{\psi},\widehat{C}}(s)\right)^{\alpha} \d s,
\end{equation*}
where $\{\tau_{n}\}_{n \in \mathbb{N}}$ is a localizing sequence of stopping times.
We can now let $n \to + \infty$ and use the first part of condition \eqref{E88} to apply dominated convergence on the left hand side. On the right hand side we can use the monotone convergence theorem. After all, we have
\begin{equation*}
  \mathbb{E}\, \e^{-\gamma t} V(z^{\widehat{\psi},\widehat{C}}(t),Y(t))=V(z,y) -\mathbb{E}\frac{1}{\alpha}\int_{0}^{t}e^{-\gamma s}\left(\widehat{C}(s)z^{\widehat{\psi},\widehat{C}}(s)\right)^{\alpha} \d s.
\end{equation*}
Consequently, by applying the second part of condition \eqref{E88}, we obtain the desired formula 
 \[
V(z,y) = \mathbb{E} \frac{1}{\alpha}\int_{0}^{+\infty}e^{-\gamma s}\left(\widehat{C}(s)z^{\widehat{\psi},\widehat{C}}(s)\right)^{\alpha} \d s.
 \]
\end{proof}

\begin{example}\textbf{(Vasicek model)} In the Vasicek model (see Section \ref{SVasicek}), $G(t,r)$ is given by \eqref{E72} and consequently  
\begin{equation} \label{E810}
G(r)=\int_{0}^{+\infty} \e^{ m_{1,0,s} +m_{2,0,s} r +\frac{1}{2} \sigma_{0,s}^{2}} \d s \, ,
\end{equation}
where $\sigma_{0,s}$, $m_{1,0,s}$, and $m_{1,0,s}$ were defined in Proposition \ref{PVasicek}.  To ensure convergence of the integral in \eqref{E810} we have to assume
\[
\int_{0}^{+\infty} \e^{m_{1,0,s}  +\frac{1}{2} \sigma_{0,s}^{2}}    \d s < + \infty.
\]
It is not difficult to find a sufficient condition for the coefficients of the model to ensure such convergence. We left the problem to the reader. However, it should be noted that the coefficient $m_{2,0,s}$ is uniformly bounded and therefore the fraction $\frac{D_rG}{G}$ is uniformly bounded as well.
\end{example}

\begin{example}\textbf{(CIR model)} In the CIR model (see Section \ref{SCIR}), $G(t,r)$ is given by \eqref{E74} and consequently 
$G(r)= \int_{0}^{+\infty} \e^{m_{1,0,s} +m_{2,0,s} r } \d s$. 
Obviously, we should require that $\int_{0}^{+\infty} \e^{m_{1,0,s}}\d s < +\infty$.  Note that  under this assumption the quotient  $\frac{D_{r}G}{G}$ is uniformly  bounded.
\end{example} 
 
 \appendix
\section{Short introduction to HJM model}\label{Appendix}
 
 Let us denote by $P(t,S)$ the price at time $t$ of a bond paying $1$ at time $S$. Assume that  the \emph{forward rates}
$f(t,S)= -\frac{\partial }{\partial S} \log P(t,S)$, $0\le t\le S$, are given by the It\^o equation 
$$
\d f(t,S)= \mu(t,S)\d t +\langle \xi(t,S), \d W(t)\rangle , 
$$ 
where $W=(W_1,\ldots, W_m)$ is an $m$ dimensional  Wiener process defined on a filtered probability space $(\Omega,\mathfrak{F},(\mathfrak{F}_t),\mathbb{P})$,  $\mu$ and $\xi$ are $\mathbb{R}$  and $\mathbb{R}^{m}$ valued processes which may depend on the forward rate $f$, and $\langle\cdot,\cdot \rangle$ denotes the Euclidean scalar product. We denote by $\|\cdot \|$ the corresponding norm. Clearly $P(t,S)=\e^{-\int_t^S f(t,u)\d u}$.  The so-called \emph{short-rate} process $\overline{r}(t):=f(t,t)$ defines the bank rate at time $t$. 

Let $T\in (0,+\infty)$ be a finite time horizon. It is well-known, see e.g. Barski and Zabczyk \cite{Barski-Zabczyk} or the original Heath, Jarrow and Morton paper \cite{HJM},  that the model  $P(t,S)$, where $t\in [0,T]$ and $t\le S<+\infty$,  is free of arbitrage if and only if there is an adapted process $\lambda$ such that  
$$
\mathbb{P}\left(\int_0^T\| \lambda(u)\|^2  \d u<+\infty\right)=1,\quad \mathbb{E}\, \mathcal{E}(\lambda)=1, 
$$
where 
$$
\mathcal{E}\,(\lambda):=  \e^{-\int_0^T\langle \lambda(u),\d W(u)\rangle -\frac 12 \int_0^T \|\lambda (u)\|^2\d u}, 
$$
and the following \emph{HJM condition is satisfied}
$$
\mu(t,S)=\left\langle \xi(t,S),  \int_t^S\xi(t,u)\d u +\lambda(t)\right\rangle, \qquad \forall\, 0\le t\le T,\quad \forall\, t\le S.  
$$
Recall that $\d \mathbb{P}^*= \mathcal{E}(\lambda)\d \mathbb{P}$ is the \emph{martingale measure}; the discounted prices $P(t,S)\e^{-\int_0^t \overline {r}(s)\d s}$, $S\le T$,  are local martingales with respect to $\mathbb{P}^*$. Moreover,   $W^*(t)= W(t)+ \int_0^t \lambda(u)\d u$ is a Wiener process with respect to $\mathbb{P}^*$.

For our purposes it is convenient to rewrite the prices and forward rates in the so-called \emph{Musiela parametrization}
$$
P(t)(x):= P(t,t+x), \quad r(t)(x):=f(t,t+x), \qquad x\ge 0.  
$$
Then 
$$
P(t)(x)= \e^{-\int_0^x r(t)(u)\d u}, \qquad r(t)(x)=-\frac{\partial }{\partial x} \log P(t)(x), 
$$
the short rate is given by the formula  $\overline{r}(t):= r(t)(0)$ and 
\begin{multline*}
r(t)(x)= r(0)(t+x)+ \int_0^t b(s)(x+t-u)\d u \\
+ \int_0^t \langle \sigma(s)(x+t-u),\d W(u)\rangle, 
\end{multline*}
where 
$$
b(t)(x):= \mu(t,t+x), \quad \sigma(t)(x):= \xi(t,t+x),\qquad x\ge 0. 
$$
Note that the HJM  conditions has the form   
$$
\begin{aligned}
b(t)(x)&=\mu(t,t+x)= \left\langle \xi(t,t+x), \int_t^{t+x}\xi(t,u)\d u +\lambda(t)\right\rangle\\
&= \left\langle \sigma(t)(x),  \int_0^x \sigma(t)(y)\d y +\lambda(t)\right\rangle \\
&=  \left\langle \sigma(t)(x),  \tilde {\sigma}(t)(x) +\lambda(t)\right\rangle,
\end{aligned}
$$
where 
$\tilde \sigma(t)(x):= \int_0^x \sigma(t)(u)\d u$.  In general $b$ and $\sigma$ may depend on $r$, 

Informally,   $S(t)\psi(x)=\psi(x+t)$ is the semigroup generated by the derivative operator $\frac {\partial }{\partial x}$. Thus $r$ is the so-called mild solution to the following stochastic partial differential equation 
$$
\d r = \left( \frac {\partial r}{\partial x}+ b\right)\d t +\langle \sigma , \d W\rangle. 
$$

We can now compute the stochastic derivative of 
$$
P(t)(x)= \e^{-\int_0^x r(t)(u)\d u}, \qquad t\ge 0.
$$
We have 
\begin{align*}
\d P(t)(x)&= P(t)(x)\left[ -\d \int_0^x r(t)(u)\d u + \frac 12 \left\|\int_0^x \sigma(t)(u)\d u\right\|^2\d t\right]. 
\end{align*}
Since, by the HJM condition,
\begin{align*}
&- \int_0^x b(t)(u)\d u + \frac 12 \left\|\int_0^x \sigma(t)(u)\d u\right\|^2 \\
&\quad = - \int_0^x \left\langle  \sigma(t)(u), \int_0^u \sigma(t)(y)\d y +\lambda(t)\right\rangle \d u + \frac 12 \left\|\int_0^x \sigma(t)(u)\d u\right\|^2 \\
&\quad = - \left\langle \lambda(t),\int_0^x \sigma(t)(u)\d u\right\rangle, 
\end{align*}
we eventually have 
\begin{align*}
&\frac{\d P(t)(x)}{P(t)(x)} \\ &\quad = - \int_0^x \frac{\partial r}{\partial u}(t)(u)\d u \d t - \left\langle\int_0^x \sigma(t)(u)\d u,   \lambda(t)\d t + \d W(t)\right\rangle\\
&\quad = \left[ -  r(t)(x)+ \overline r(t)\right] \d t   -  \left\langle \tilde \sigma(t)(x), \lambda(t)\d t +\d W(t)\right\rangle. 
\end{align*}

The instrument $P(t)(x)$, $t\ge 0$,  is called a \emph{sliding bond}.  After discounting it by a bank account we get
$$
\frac{\d \bar{P}(t)(x)}{\bar{P}(t)(x)} =  -  r(t)(x) \d t   -  \left\langle \tilde \sigma(t)(x), \lambda(t)\d t +\d W(t)\right\rangle. 
$$
The discounted sliding bonds  are  not  tradable instruments. To work with a portfolio process we will use the so-called \emph{rolling bonds}, see \eqref{E12}.  

\medskip
\noindent
\textbf{ Acknowledgements}. The authors would like to thank  anonymous referees  for careful reading of the paper and for providing many useful remarks and comments which have led to a substantially improved and more complete presentation

\end{document}